\documentclass[a4paper,12pt]{extarticle}

\usepackage{latexsym}
\usepackage{amscd}
\usepackage{graphics}
\usepackage{amsmath}
\usepackage{amssymb}
\usepackage{mathrsfs}
\usepackage{amsthm}
\usepackage{xcolor}
\usepackage{accents}
\usepackage{enumerate}
\usepackage{url}
\usepackage{tikz-cd}
\usepackage{marginnote}
\usepackage{hyperref}
\usepackage{soul}
\usepackage[many]{tcolorbox}

\usepackage{newpxtext} \usepackage[euler-digits]{eulervm}

\usepackage{import}
\usepackage{xifthen}
\usepackage{pdfpages}
\usepackage{transparent}
\usepackage{testhyphens}

\theoremstyle{plain}
\newtheorem{theorem}{\sc Theorem}[section]
\makeatletter
\newcommand{\settheoremtag}[1]{
  \let\oldthetheorem\thetheorem
  \renewcommand{\thetheorem}{#1}
  \g@addto@macro\endtheorem{
    \addtocounter{theorem}{-1}
    \global\let\thetheorem\oldthetheorem}
  }
\newtheorem{prop}[theorem]{\sc Proposition}
\newtheorem{lem}[theorem]{\sc Lemma}
\newtheorem{cor}[theorem]{\sc Corollary}
\theoremstyle{definition}
\newtheorem{defn}[theorem]{\sc Definition}
\newtheorem{rem}[theorem]{\sc Remark}

\newtheorem{ex}[theorem]{\sc Example}
\newtheorem{taggedtheoremx}{\sc Proposition}
\newenvironment{taggedtheorem}[1]
 {\renewcommand\thetaggedtheoremx{#1}\taggedtheoremx}
 {\endtaggedtheoremx}

\DeclareMathOperator{\R}{\mathbb{R}}
\newcommand{\abs}[1]{\left\lvert#1\right\rvert}
\newcommand{\norm}[1]{\left\lVert#1\right\rVert}
\newcommand{\op}[1]{\operatorname{#1}}
\renewcommand{\hat}[1]{\widehat{#1}}
\numberwithin{equation}{section}

\tcbset{
  myhlight/.style={
    colback=yellow,
    arc=0pt,
    outer arc=0pt,
    boxrule=0pt,
    top=0pt,
    bottom=0pt,
    left=0pt,
    right=0pt,
  },
  highlight math style={myhlight}
}
\newtcolorbox{myhl}{breakable, myhlight}

\title{Exotic symplectomorphisms and contact circle actions}
\date{\today}
\author{Du\v{s}an Drobnjak\footnote{Faculty of mathematics, University of Belgrade, Studentski trg 16, 11158 Belgrade, Serbia;\:\:\: e-mail: dusan\_drobnjak@matf.bg.ac.rs}\hspace{0.2cm} and Igor Uljarevi\'{c}\footnote{Faculty of mathematics, University of Belgrade, Studentski trg 16, 11158 Belgrade, Serbia;\:\:\: e-mail: igoru@matf.bg.ac.rs}}

\usepackage{biblatex}
\begin{document}
\maketitle{}

\begin{abstract}
Using Floer-theoretic methods, we prove that the non-existence of an exotic symplectomorphism on the standard symplectic ball, $\mathbb{B}^{2n},$ implies a rather strict topological condition on the free contact circle actions on the standard contact sphere, $\mathbb{S}^{2n-1}.$ We also prove an analogue for a Liouville domain and contact circle actions on its boundary. Applications include results concerning the symplectic mapping class group and the fundamental group of the group of contactomorphisms. 
\end{abstract}

\section{Introduction}

A well-known open problem in symplectic topology is that of the existence of the so-called exotic symplectomorphisms on the standard $2n$-dimensional symplectic ball, $\mathbb{B}^{2n}.$ A compactly supported symplectomorphism $\mathbb{B}^{2n}\to \mathbb{B}^{2n}$ is called exotic if it is not isotopic to the identity relative to the boundary through symplectomorphisms. Apart from the cases where $n=1$ and $n=2$ (where the topology of the group of compactly supported symplectomorphisms $\mathbb{B}^{2n}\to \mathbb{B}^{2n}$ is well understood due to Gromov's theory of $J$-holomorphic curves \cite{gromov1985pseudo}) very little is known about exotic symplectomorphisms of $\mathbb{B}^{2n}.$ For instance, it is an open problem whether an exotic diffeomorphism of $\mathbb{B}^{2n}$ can be realized as a symplectomorphism with respect to the standard symplectic structure on $\mathbb{B}^{2n}$ (compare to \cite{casals2018symplectomorphisms} where such a realization is constructed in the case of a non-standard symplectic structure on a ball). In other words, it is not known whether the problem of the existence of the exotic symplectomorphisms on the standard symplectic ball can be solved in the framework of differential topology.

On the other hand, there are numerous results regarding exotic symplectomorphisms on other symplectic manifolds \cite{seidel1997floer, seidel2008lectures, chiang2014open, seidel2015exotic, chiang2016non, tonkonog2016commuting, shevchishin2017elliptic, uljarevic2017floer, barth2019diffeomorphism, uljarevic2019viterbo}. The most notable is work of Seidel \cite{seidel1997floer} who constructed the first known example of a symplectomorphism that is smoothly isotopic to the identity but not symplectically, thus proving that exotic symplectomorphisms are genuinely a symplectic phenomenon.

In the present paper, we prove that non-existence of an exotic symplectomorphism on the standard $\mathbb{B}^{2n}$ implies a rather strict topological condition on the free contact circle actions on its boundary (i.e. the standard contact sphere $\mathbb{S}^{2n-1}$). We call this condition \emph{topological symmetry.} It is expressed in terms of the reduced homology (denoted by $\tilde{H}_\ast$ below) of the part of $\mathbb{S}^{2n-1}$ on which the circle action is positively transverse to the contact distribution.

\begin{defn}
Let $\xi$ be a cooriented contact distribution on $\mathbb{S}^{2n-1},$ and let $\varphi_t:(\mathbb{S}^{2n-1},\xi)\to (\mathbb{S}^{2n-1},\xi)$ be a free contact circle action. Denote by $Y$ the vector field that generates $\varphi_t,$ and by $P\subset \mathbb{S}^{2n-1}$ the set of points in $\mathbb{S}^{2n-1}$ at which the vector field $Y$ represents the positive coorientation of $\xi.$ The contact circle action $\varphi$ is said to be \emph{topologically symmetric} if there exists $m\in\mathbb{Z}$ such that
\[ \dim \tilde{H}_{m-k}(P;\mathbb{Z}_2)= \dim \tilde{H}_{k}(P;\mathbb{Z}_2)\]
for all $k\in\mathbb{Z}.$
\end{defn}

\begin{theorem}\label{thm:alternatives}
For all $n\in\mathbb{N},$ at least one of the following statements is true.
\begin{enumerate}[A]
    \item There exists an exotic symplectomorphism of $\mathbb{B}^{2n}$ .
    \item Every free contact circle action on $\mathbb{S}^{2n-1}$ is topologically symmetric.
\end{enumerate}
\end{theorem}

In fact, we prove a more general statement about exotic symplectomorphisms of Liouville domains and contact circle actions on their boundaries (see Theorem~\ref{thm:main} on page~\pageref{thm:main}). A consequence of this more general result is that topologically asymmetric free contact circle actions represent nontrivial elements in the fundamental group of the group of contactomorphisms (for a precise statement see Corollary~\ref{cor:fundgroup} on page~\pageref{cor:fundgroup}).

Theorem~\ref{thm:alternatives} together with the result of Gromov that there are no exotic symplectomorphisms on $\mathbb{B}^4$ implies the following.

\begin{cor}
Every free contact circle action on the standard contact $\mathbb{S}^3$ is topologically symmetric.
\end{cor}

The paper is organized as follows. Section~\ref{sec:prel} recalls the preliminaries. Section~\ref{sec:topsym} introduces the notion of \emph{topological symmetry} and proves some of the basic properties of the free contact circle actions. Section~{\ref{sec:app}} proves Theorem~{\ref{thm:alternatives}}. Section~\ref{sec:morse} contains the main technical result, that relies on the Morse theory on a manifold with boundary. 

\subsection*{Acknowledgements}
\sloppy We would like to thank Paul Biran, Aleksandra Marinkovi\'{c}, Darko Milinkovi\'{c}, Vuka\v{s}in Stojisavljevi\'{c}, and  Filip \v{Z}ivanovi\'{c}  for useful feedback. This work was partially supported by the Ministry of Education, Science, and Technological development, grant number 174034.
\section{Preliminaries}\label{sec:prel}

\subsection{Notation and conventions}

Let $(W,\omega)$ be a symplectic manifold. The Hamiltonian vector field, $X_{H_t},$ of the Hamiltonian $H_t: W\to\R$ is the vector field on $W$ defined by $\omega(X_{H_t},\cdot)=dH_t.$ We denote by $\phi^H_t:W\to W$ the Hamiltonian isotopy of the Hamiltonian $H_t:W\to\R,$ i.e. $\partial_t\phi^H_t= X_{H_t}\circ\phi^H_t$ and $\phi_0^H=\op{id}.$

Let $\Sigma$ be a contact manifold with a contact form $\alpha.$ The Reeb vector field $R^\alpha$ is the unique vector filed on $\Sigma$ such that $\alpha(R^\alpha)=1$ and such that $d\alpha(R^\alpha, \cdot)=0$. We denote by $\varphi_t^h:\Sigma\to \Sigma$ the contact isotopy furnished by a contact Hamiltonian $h_t:\Sigma\to\R.$ In our conventions, the vector field $Y$ of the isotopy $\varphi^h_t$ and the contact Hamiltonian $h$ are related by $h=-\alpha(Y).$ A contact Hamiltonian $h_t:\Sigma\to \R$ is called strict (with respect to the contact form $\alpha$) if $dh_t(R^\alpha)=0.$ This is equivalent to $\left(\varphi_t^h\right)^\ast \alpha=\alpha$, for all $t\in\R$.

\subsection{Liouville domains}

\begin{defn}
A Liouville domain is a compact manifold $W$ together with a 1-form $\lambda$ such that the following holds.
\begin{itemize}
    \item The 2-form $d\lambda$ is a symplectic form on $W.$
    \item The restriction of $\lambda$ to the boundary $\partial W$ is a contact form that induces the boundary orientation on $\partial W.$ 
\end{itemize}
\end{defn}

A part of the symplectization of $\partial W$ naturally embeds into $W.$ More precisely, there exists a unique embedding $\iota :\partial W\times(0,1]\to W$ such that $\iota(x,1)=x,$ and such that $\iota^\ast \lambda= r\cdot \alpha.$ Here, $r$ is the $(0,1]$ coordinate function and $\alpha:=\left.\lambda\right|_{\partial W}$ is the contact form on $\partial W.$ The completion, $\hat{W},$ of the Liouville domain $W$ is obtained by gluing $W$ and $\partial W\times (0,\infty)$ via $\iota.$ The sets $W$ and $\partial W\times(0,\infty)$ can be seen as subsets of the completion, $\hat{W}.$ The manifold $\hat{W}$ is an exact symplectic manifold with a Liouville form given by
\[ \hat{\lambda}:=\left\lbrace \begin{matrix} \lambda & \text{on }W,\\ r\cdot\alpha&\text{on } \partial W\times (0,\infty).  \end{matrix}\right. \]
With a slight abuse of notation, we will write $\lambda$ instead of $\hat{\lambda}.$

\subsection{A homotopy long exact sequence}\label{sec:les}

In this section, we describe a construction that is due to Biran and Giroux \cite{biran2005symplectic}. Given a Liouville domain $(W,\lambda),$ denote by $\Sigma$ the manifold $\partial W,$ and by $\alpha$ the contact form $\left.\lambda\right|_{\partial W}$ on $\Sigma.$ Denote by $\op{Cont}\Sigma, \op{Symp}_cW, \op{Symp}(W,\lambda), $ and $\op{Symp}_c(W,\lambda)$ the groups of diffeomorphisms defined in the list below.

\begin{description}
\item[$\op{Cont} \Sigma$] is the group of the contactomorphisms of $(\Sigma,\ker \alpha)$.
\item[$ \op{Symp}_cW$] is the group of the symplectomorphisms of $(W,d\lambda)$ that are equal to the identity in a neighbourhood of the boundary.
\item[$\op{Symp}(W,\lambda)$] is the group of the so-called exact symplectomorphisms of $(\hat{W}, d\lambda)$ that preserve the Liouville form $\lambda$ in $\Sigma\times[1,\infty).$ More precisely, a symplectomorphism $\phi:\hat{W}\to\hat{W}$ is an element of $\op{Symp}(W,\lambda)$ if, and only if, there exists a smooth function $F:\hat{W}\to\R$ such that $\left.F\right|_{\Sigma\times[1-\varepsilon,\infty)}\equiv 0,$ for some $\varepsilon>0,$ and such that $\phi^\ast \lambda-\lambda = dF.$
\item[$\op{Symp}_c(W,\lambda)$] is the group of the symplectomorphisms in $\op{Symp}(W,\lambda)$ that are equal to the identity on $\Sigma\times(1-\varepsilon,\infty)$ for some $\varepsilon>0.$
\end{description}

Since a symplectomorphism $\phi\in\op{Symp}(W,\lambda)$ preserves the Liouville form $\lambda$ on the cylindrical end $\Sigma\times [1,\infty),$ there exists a contactomorphism $\varphi:\Sigma\to\Sigma$ such that $\phi$ has the following form on the cylindrical end
\[\phi(x,r)=\left( \varphi(x), ? \right) \in\Sigma\times(0,\infty)\]
for all $(x,r)\in\Sigma\times[a,\infty),$ where $a\in\R^+$ is big enough. The map $\op{Symp}(W,\lambda)\to \op{Cont}\Sigma$ defined by $\phi\mapsto\varphi$ is called the ideal restriction map. It turns out that the ideal restriction map is a Serre fibration with the fibre above the identity equal to $\op{Symp}_c(W,\lambda).$ Hence, there is a homotopy long exact sequence
\[\begin{tikzcd}[column sep=small]
\cdots \arrow{r} & \pi_k\op{Symp}_c(W,\lambda) \arrow{r}
& \pi_k\op{Symp}(W,\lambda) \arrow{r}
\arrow[draw=none]{d}[name=Z, shape=coordinate]{}
& \pi_k\op{Cont}\Sigma \arrow[rounded corners,
to path={ -- ([xshift=2ex]\tikztostart.east)
|- (Z) [near end]\tikztonodes
-| ([xshift=-2ex]\tikztotarget.west)
-- (\tikztotarget)}]
{dll}{ } \\
 \arrow[draw=none]{r} & \pi_{k-1}\op{Symp}_c(W,\lambda)\arrow{r}
& \cdots.
\end{tikzcd}\]
The groups $\op{Symp}_c(W,\lambda)$ and $\op{Symp}_c W$ are homotopy equivalent \cite{biran2005symplectic} (see also \cite[Lemma~3.3]{uljarevic2017floer}). Therefore, there is a homotopy long exact sequence
\[\begin{tikzcd}[column sep=small]
\cdots \arrow{r} & \pi_k\op{Symp}_c W  \arrow{r}
& \pi_k\op{Symp}(W,\lambda) \arrow{r}
\arrow[draw=none]{d}[name=Z, shape=coordinate]{}
& \pi_k\op{Cont}\Sigma \arrow[rounded corners,
to path={ -- ([xshift=2ex]\tikztostart.east)
|- (Z) [near end]\tikztonodes
-| ([xshift=-2ex]\tikztotarget.west)
-- (\tikztotarget)}]
{dll}{ } \\
 \arrow[draw=none]{r} & \pi_{k-1}\op{Symp}_c W\arrow{r}
& \cdots.
\end{tikzcd}\]
Particularly important in this paper is the boundary map
\[\Theta: \pi_1\op{Cont}\Sigma\to \pi_0\op{Symp}_c W\]
from the exact sequence above. The map $\Theta$ can be described as follows. Given a loop 
\[\varphi:\mathbb{S}^1\to \op{Cont}\Sigma :t\mapsto \varphi_t.\]
Let $h_t:\Sigma\to\R$ be a contact Hamiltonian that generates it. Choose a Hamiltonian $H_t:\hat{W}\to\R$ such that 
\[H_t(x,r)=r\cdot h_t(x)\quad\text{for}\quad (x,r)\in\Sigma\times (1-\varepsilon,\infty)\]
for some $\varepsilon>0.$ Then, $\Theta([\varphi])=\left[ \phi^H_1 \right].$

\subsection{Floer theory}

The Floer homology that is utilized in this paper is the Floer homology for contact Hamiltonians. It has been introduced by Merry and the second author in \cite{merry2019maximum} as a consequence of a generalized no-escape lemma. Our results, however, use only the case of strict contact Hamiltonians (i.e. contact Hamiltonians whose isotopies preserve not only the contact distribution but also the contact form), the case considered already in \cite{ritter2016circle}, \cite{Fauck2016Rabinowitz-Floer} and \cite{uljarevic2017floer}.

The Floer homology for a contact Hamiltonian, $HF_\ast(h),$ is associated to a contact Hamiltonian $h:\partial W\times \mathbb{S}^1\to \R$ defined on the boundary of a Liouville domain $W$ such that the time-one map of $h$ has no fixed points. By definition, $HF_\ast(h)$ is equal to the Hamiltonian loop Floer homology for a Hamiltonian $H:\hat{W}\times \mathbb{S}^1\to\R$ such that $H_t(x,r)=r\cdot h_t(x)$ whenever $(x,r)\in\partial W\times [1,\infty).$ 

\subsubsection{Floer data}

To define $HF_\ast(h),$ one has to choose auxiliary data consisting of a Hamiltonian $H: \hat{W}\times\mathbb{S}^1\to \R$ and an $\mathbb{S}^1$ family $\{J_t\}_{t\in\mathbb{S}^1}$ of almost complex structures on $\hat{W}$ that satisfy the following conditions.

\begin{enumerate}
    \item Conditions on the cylindrical end.
    \begin{itemize}
        \item $H_t(x,r)=r\cdot h_t(x),$ for $(x,r)\in \partial W\times[1,+\infty),$
        \item $dr\circ J_t=-\lambda,$ in $ W\times[1,+\infty).$
    \end{itemize}
    \item Non-degeneracy. For each fixed point $x$ of $\phi_1^H:\hat{W}\to\hat{W},$
    \[\det \left( d\phi_1^H(x)-\op{id} \right)\not=0.\]
    \item $d\lambda$-compatibility. $d\lambda(\cdot, J_t\cdot)$ is an $\mathbb{S}^1$ family of Riemannian metrics on $\hat{W}.$
    \item Regularity. The linearized operator of the Floer equation
    \begin{align*}
        & u:\R\times\mathbb{S}^1\to \hat{W},\\
        & \partial_s u + J_t(u) (\partial_t u - X_{H_t}(u))=0
    \end{align*}
    is surjective.
\end{enumerate}

\subsubsection{Floer complex}

The Floer complex, $CF_\ast(H,J),$ is generated by the contractible 1-periodic orbits of the Hamiltonian $H.$ I.e.
\[ CF_k(H,J):=\bigoplus_{\deg \gamma=k}\mathbb{Z}_2\left\langle \gamma\right\rangle, \]
where the sum goes through the set of the contractible 1-periodic orbits of $H$ that have the degree equal to $k.$ The degree, $\deg \gamma,$ is defined as the negative Conley-Zehnder index of the path of symplectic matrices obtained from $d\phi_t^H(\gamma(0))$ by trivializing $T\hat{W}$ along a disk that bounds $\gamma.$ Due to different choices of the disk that bounds $\gamma,$ $\deg\gamma$ is only defined up to a multiple of $2N,$ where
\[N:=\min \left\{c_1(u)>0 \quad |\quad  u:\mathbb{S}^2\to\hat{W}\right\}\]
is the minimal Chern number. Hence we see $\deg\gamma$ as an element of $\mathbb{Z}_{2N},$ and the Floer complex $CF_\ast(H,J)$ is $\mathbb{Z}_{2N}$ graded. The differential $\partial : CF_{k+1}(H,J)\to CF_k(H,J)$ is defined by counting the isolated unparametrized solutions of the Floer equation. More precisely,
\[\partial \left\langle\gamma\right\rangle =\sum n(\gamma,\tilde{\gamma})\left\langle \Tilde{\gamma} \right\rangle,\]
where $n(\gamma,\tilde{\gamma})$ is the number modulo 2 of the isolated unparametrized solutions $u: \mathbb{R}\times\mathbb{S}^1\to \hat{W}$ of the following problem
\begin{align*}
    & \partial_su+ J_t(u)(\partial_tu-X_{H_t}(u))=0,\\
    &\lim_{s\to-\infty}u(s,t)=\gamma(t),\\
    &\lim_{s\to+\infty}u(s,t)=\tilde{\gamma}(t).
\end{align*}
Note that, if $\deg\tilde{\gamma}\not=\deg\gamma-1,$ then there are no isolated unparametrized solutions of the Floer equation from $\gamma$ to $\tilde{\gamma}.$ Hence, in this case, $n(\gamma,\tilde{\gamma})=0,$ and $\partial$ is well defined. The Floer homology, $HF_\ast(H,J),$ is the homology of the Floer complex $CF_\ast(H,J).$

\subsubsection{Continuation maps}

Given Floer data $(H^-,J^-)$ and $(H^+, J^+).$ Continuation data from $(H^-,J^-)$ to $(H^+, J^+)$ consists of a ($s$-dependent) Hamiltonian $H:\hat{W}\times \R\times\mathbb{S}^1\to \R$ and a family $\{J_{s,t}\}_{(s,t)\in\R\times\mathbb{S}^1}$ of almost complex structures on $\hat{W}$ such that
\begin{enumerate}
    \item \label{item:slope}there exists a smooth function $h:\partial W\times \R\times\mathbb{S}^1\to \R$ that is non-decreasing in $\R$-variable and such that
    \[H_{s,t}(x,r):=H((x,r),s,t)= r\cdot h(x,s,t)=: r\cdot h_{s,t}(x)\]
    for all $(x,r)\in\partial W\times [1,\infty),$
    \item $dr\circ J_{s,t}=-\lambda$ in $\partial W\times [1,\infty),$ for all $(s,t)\in\R\times\mathbb{S}^1,$
    \item $d\lambda(\cdot, J_{s,t}\cdot)$ is a Riemannian metric on $\hat{W}$ for all $(s,t)\in\R\times\mathbb{S}^1,$
    \item $(H_{s,t}, J_{s,t})=(H^\pm_t,J^\pm_t)$ for $\pm s>> 0.$
\end{enumerate}

Let $\gamma^-$ be a $1$-periodic orbit of $H^-,$ and let $\gamma^+$ be a 1-periodic orbit of $H^+.$ For generic continuation data $(H_{s,t}, J_{s,t})$ from $(H^-,J^-)$ to $(H^+,J^+),$ the set of the solutions $u:\R\times\mathbb{S}^1\to\hat{W}$ of the problem
\begin{align*}
    & \partial_s u+ J_{s,t}(u)(\partial_t u- X_{H_{s,t}}(u))=0,\\
    &\lim_{s\to\pm\infty}u(s,t)=\gamma^{\pm}(t)
\end{align*}
is a finite union of compact manifolds (possibly of different dimensions) cut out transversely by the Floer equation. Denote by $n(\gamma^-,\gamma^+)$ the number of its 0-dimensional components. The continuation map
\[\Phi=\Phi(\{H_{s,t}\}, \{J_{s,t}\}): CF_\ast(H^-,J^-)\to CF_\ast(H^+, J^+)\]
is the linear map defined on the generators by
\[\Phi(\gamma^-):=\sum_{\gamma^+}n(\gamma^-,\gamma^+)\left\langle \gamma^+ \right\rangle.\]
Since there are no 0-dimensional components of the above mentioned manifold if $\deg\gamma^-\not=\deg \gamma^+,$ continuation maps preserve the grading. By the condition~\ref{item:slope} for continuation data, continuation maps  
\[CF_\ast(H^-, J^-)\to CF_\ast(H^+,J^+)\]
are defined only if $H^-\leqslant H^+$ on the cylindrical end $\partial W\times[1,\infty).$ The continuation maps $CF_\ast(H^-, J^-)\to CF_\ast(H^+,J^+)$ defined with respect to different continuation data form $(H^-,J^-)$ to $(H^+,J^+)$ are chain homotopic. Hence, they induce the same map $HF_\ast(H^-, J^-)\to HF_\ast(H^+,J^+)$ on the homology level. The induced map is also called the continuation map. As opposed to the compact case, the continuation maps need not be isomorphisms. They do, however, satisfy the following relations.
\begin{enumerate}
    \item The continuation map $\Phi_\alpha^\alpha: HF_\ast(H^\alpha, J^\alpha)\to HF_\ast(H^\alpha,J^\alpha)$ is equal to the identity.
    \item The composition of the continuation maps 
    \[\Phi_\alpha^\beta: HF_\ast(H^\alpha, J^\alpha)\to HF_\ast(H^\beta,J^\beta)\] 
    and 
    \[\Phi_\beta^\gamma: HF_\ast(H^\beta, J^\beta)\to HF_\ast(H^\gamma,J^\gamma)\] 
    is equal to the continuation map 
    \[\Phi_\alpha^\gamma: HF_\ast(H^\alpha, J^\alpha)\to HF_\ast(H^\gamma,J^\gamma),\] 
    i.e. $\Phi^\gamma_\beta\circ\Phi_\alpha^\beta=\Phi_\alpha^\gamma.$
\end{enumerate}
In other words, the family of groups $\{HF_\ast(H,J)\}$ together with the continuation maps form a directed system of groups. As a consequence, the groups $HF_\ast(H^\alpha, J^\alpha)$ and $HF_\ast(H^\beta, J^\beta)$ are canonically isomorphic if $H^\alpha= H^\beta$ on $\partial W\times[1,\infty).$ Therefore, the group $HF_\ast(h),$ where $h:\partial W\times\mathbb{S}^1\to\R$ is a 1-periodic contact Hamiltonian whose time-1 map has no fixed points, is well defined. Moreover, if $h^\alpha, h^\beta:\partial W\times\mathbb{S}^1\to\R$ are two 1-periodic contact Hamiltonians whose time-1 maps have no fixed points and such that $h^\alpha\leqslant h^\beta,$ then there is a well defined continuation map $HF_\ast(h^\alpha)\to HF_\ast(h^\beta).$

\subsubsection{Naturality isomorphisms}

Let $H, F$ be Hamiltonians, denote by $\overline{H}$ and $H\# F$ the Hamiltonians that generate Hamiltonian isotopies $\left(\phi^H_t\right)^{-1}$ and $\phi_t^H\circ\phi^F_t,$ respectively. The naturality isomorphism
\[\mathcal{N}(F): HF(H)\to HF(\overline{F}\# H)\]
is associated to a $1$-periodic Hamiltonian $F:\hat{W}\times\mathbb{S}^1\to\R$ whose Hamiltonian isotopy is 1-periodic (i.e. $F$ generates a loop of Hamiltonian diffeomorphisms). On the chain level, on generators, the map $\mathcal{N}(F)$ is defined by
\[ \left\langle \gamma\right\rangle\mapsto \left\langle (\phi^F)^\ast\gamma \right\rangle, \]
where $(\phi^F)^\ast\gamma :\mathbb{S}^1\to\hat{W}$ is the loop $t\mapsto (\phi^F_t)^{-1}(\gamma(t)).$ The naturality map is an isomorphism already on the chain level. The naturality isomorphisms do not preserve the grading in general. They do respect the grading up to a constant shift though.

\section{Topologically symmetric contact circle actions}\label{sec:topsym}

A contact circle action is a Lie group action of $\mathbb{S}^1$ on a contact manifold by contactomorphisms. A contact circle action on $\Sigma$ can be seen as a 1-periodic family $\varphi_t:\Sigma\to\Sigma, t\in\mathbb{R}$ of contactomorphisms such that
\[\varphi_s\circ\varphi_t=\varphi_{s+t},\]
for all $s, t\in\R.$ In particular, every contact circle action on $\Sigma$ can be seen as a flow of an autonomous vector field on $\Sigma.$

\begin{defn}\label{def:topsym}
Let $\varphi_t:\Sigma\to\Sigma$ be a contact circle action on a (cooriented) contact manifold $(\Sigma,\xi)$ that is the boundary of a Liouville domain $W$ with $c_1(W)=0.$ Let $h:\Sigma\to\R$ be the contact Hamiltonian that generates $\varphi_t$ defined with respect to some contact form on $\Sigma.$ Denote
\[ \Sigma^+:=\left\{ p\in\Sigma\:|\: h(p)>0 \right\}. \]
The contact circle action $\varphi_t$ is called \emph{topologically symmetric} (with respect to the filling $W$) if there exists an integer $m\in\mathbb{Z}$ such that
\[(\forall k\in\mathbb{Z})\quad \dim H_{m-k}(W,\Sigma^+;\mathbb{Z}_2)= \dim H_{k}(W,\Sigma^+;\mathbb{Z}_2).\]
Here, $H_k(W,\Sigma^+;\mathbb{Z}_2)$ stands for the singular homology of the pair $(W,\Sigma^+)$ in $\mathbb{Z}_2$ coefficients. The set $\Sigma^+$ is referred to as the \emph{positive region} of the contact circle action $\varphi.$ Similarly, the \emph{negative region} of the contact circle action $\varphi$ is the set
\[\Sigma^-:=\left\{ p\in\Sigma\:|\: h(p)<0 \right\}. \]
\end{defn}

The set $\Sigma^+$ in the definition above (and consequently the notion of the topologically symmetric contact circle action) does not depend on the choice of the contact form on $\Sigma,$ although $h$ does. Namely, $\Sigma^+$ can be defined as the set of the points $p\in\Sigma$ such that the vector field of $\varphi_t$ at the point $p$ represents the negative coorientation of the contact distribution $\xi$ at the point $p.$ 

\begin{ex}
The Reeb flow on the standard contact sphere $\Sigma:=\mathbb{S}^{2n-1}$ is an example of a topologically symmetric contact circle action with respect to the standard symplectic ball $W:=\mathbb{B}^{2n}$. Indeed, the Reeb flow is generated by the constant contact Hamiltonian
\[\Sigma\to\R\quad:\quad x\mapsto -1.\]
Hence, the corresponding positive region $\Sigma^+$ is equal to the empty set. Consequently, $H_\ast(W,\Sigma^+;\mathbb{Z}_2)$ is isomorphic to $H_\ast(\mathbb{B}^{2n}; \mathbb{Z}_2)$. This implies the topological symmetry.
\end{ex}

\begin{ex}
Let $\varepsilon>0$ be a sufficiently small number and let $W$ be the Brieskorn variety
\[W:=\left\{ (z_0,\ldots, z_n)\in\mathbb{C}^{n+1}\:|\: z_0^3 + \cdots + z_n^3 =\varepsilon\:\&\: \abs{z}\leqslant 1 \right\}.\]
The boundary $\partial W$ is contactomorphic to the Brieskorn manifold
\[ \Sigma:=\left\{ (z_0,\ldots, z_n)\in\mathbb{C}^{n+1}\:|\: z_0^3 + \cdots + z_n^3 =0\:\&\: \abs{z}= 1 \right\}. \]
Consider the free contact circle action on $\Sigma$ given by
\[ \Sigma\times\R\quad\ni\quad (x,t) \quad\mapsto \quad e^{\frac{2\pi i t}{3}}\cdot z\quad \in \quad \Sigma. \]
\sloppy The positive region, $\Sigma^+$, is equal to the empty set. Therefore, $H_\ast(W,\Sigma^+;\mathbb{Z}_2)$ is isomorphic to $H_\ast(W;\mathbb{Z}_2)$. The Brieskorn variety $W$ is homotopy equivalent to the wedge of $2^n$ copies of $\mathbb{S}^n$. Hence, the contact circle action above is not topologically symmetric. For a detailed account on Brieskorn manifolds, see \cite{kwon2016brieskorn}.
\end{ex}

\begin{lem}\label{lem:strictcontactomorphisms}
Let $\varphi_t:\Sigma\to \Sigma$ be a contact circle action on a cooriented contact manifold $\Sigma.$ Then, there exists a contact form $\alpha$ on $\Sigma$ such that $\varphi_t^\ast\alpha=\alpha$ for all $t\in \R.$
\end{lem}
\begin{proof}
The lemma is a special case of Proposition~2.8 in \cite{lerman2001contact}.
\end{proof}

\begin{lem}\label{lem:regular}
Let $\varphi_t:\Sigma\to\Sigma$ be a free contact circle action on a (cooriented) contact manifold $\Sigma,$ and let $\alpha$ be a contact form on $\Sigma$ that is invariant under $\varphi_t.$ Denote by $h:\Sigma\to\R$ the contact Hamiltonian of $\varphi_t$ defined with respect to $\alpha.$ Then, 0 is a regular value of the function $h.$
\end{lem}
\begin{proof}
Let $Y$ be the vector field on $\Sigma$ that generates $\varphi_t,$ i.e. $\partial_t\varphi_t=Y\circ\varphi_t.$ Then, by definition, the generating contact Hamiltonian is $h=-\alpha(Y).$ The Cartan formula (together with the invariance of $\alpha$ under $\varphi_t$) implies
\[0= \frac{d}{dt}\left(\varphi^\ast_t\alpha\right)= \varphi^\ast_t\left( d\alpha(Y,\cdot) + d(\alpha(Y)) \right).\]
Hence, $dh=d\alpha(Y,\cdot).$ If $p\in h^{-1}(0),$ then $Y(p)\in\ker\alpha_p$ is a non-zero vector that belongs to the contact distribution (it is non-zero because the circle action $\varphi_t$ is a free action). Since $d\alpha$ is non-degenerate when restricted to the contact distribution, the 1-form $dh(p)=d\alpha(Y(p),\cdot)$ is non-degenerate. Therefore, $p\in h^{-1}(0)$ cannot be a critical point of $h,$ i.e. 0 is a regular value of $h.$   
\end{proof}

\begin{rem}
In the situation of Definition~\ref{def:topsym}, if 0 is a regular value of $h$, \cite[Theorem~3.43]{allen2002cambridge} implies that by replacing $\Sigma^+$ by $\Sigma^-$ in the definition one obtains an equivalent definition. More precisely, the contact circle action $\varphi$ is topologically symmetric with respect to $W$ if, and only if, there exists $m\in\mathbb{Z}$ such that
\[\dim H_{m-k}(W,\Sigma^-;\mathbb{Z}_2)= \dim H_{k}(W,\Sigma^-;\mathbb{Z}_2),\]
for all $k\in\mathbb{Z}.$
\end{rem}

\section{Topologically asymmetric contact circle actions and the topology of transformation groups}\label{sec:app}

Section~{\ref{sec:les}} discussed a method of constructing a symplectomorphism $\phi:W\to W$ of a Liouville domain from a loop of contactomorphisms $\varphi_t:\partial W\to\partial W$ of its boundary. If $h_t:\partial W\to\R$ is the contact Hamiltonian generating $\varphi_t$, then the symplectomorphism $\phi$ is obtained as the time-1 map of a Hamiltonian $H_t:W\to\R$ that is equal to $r\cdot h_t$ on the cylindrical end. The method gives rise to a homomorphism
\begin{align*}
    &\Theta \quad:\quad \pi_0\op{Cont}(\partial W)\quad\to\quad \pi_0 \op{Sump}_c W\quad :\quad [\varphi]=[\varphi^h]\mapsto [\phi^H_1]=[\phi].
\end{align*}
This section proves the main result of the paper: topologically asymmetric free contact circle actions furnish (via $\Theta$) non-trivial elements of $\pi_0 \op{Symp}_c W$. The next lemma will be used to reduce to the case where the free contact circle action preserves not only the contact distribution but also the contact form on $\partial W$.

\begin{lem}\label{lem:changealpha}
Let $(W,\lambda)$ be a Liouville domain, and let $\beta$ be a contact form on $\Sigma:=\partial W.$ Then, there exists a Liouville form $\mu$ on $W$ such that $\beta=\left. \mu \right|_\Sigma,$ and such that the following holds. If
\begin{align*}
    &\Theta^\lambda: \pi_1\op{Cont}\Sigma\to \pi_0\op{Symp}_c (W,d\lambda),\\
    &\Theta^\mu : \pi_1\op{Cont}\Sigma\to \pi_0\op{Symp}_c (W,d\mu)
\end{align*}
are the homomorphisms from Section~\ref{sec:les}, and if $\eta\in\pi_1\op{Cont}\Sigma,$ then $\Theta^\lambda(\eta)=0$ if and only if $\Theta^\mu(\eta)=0.$ 
\end{lem}
\begin{proof}
Denote $\alpha:=\left.\lambda\right|_{\Sigma}.$ Since $\alpha$ and $\beta$ determine the same (cooriented) contact structure on $\Sigma,$ there exists a positive function $f:\Sigma\to\R^+$ such that $\beta=f\alpha.$ Let $V\subset\hat{W}$ be the complement of the set
\[ \left\{ (x,r)\in\Sigma\times\R^+\subset \hat{W}\:|\: r>f(x) \right\}. \]
Since the Liouville vector field $X_\lambda$ (defined by $\lambda=X_\lambda\lrcorner d\lambda$) is nowhere vanishing in $\Sigma\times\R^+,$ and since it is transverse to both $\partial V$ and $\partial W,$ the manifolds $V$ and $W$ are diffeomorphic. Denote by $\Psi: W\to V$ the diffeomorphism furnished by $X_\lambda.$ For $x\in \partial W,$ $\Psi(x)=(x, f(x)).$ Hence, the one-form $\mu:=\Psi^\ast \lambda$ satisfies $\left.\mu\right|_{\Sigma}=\beta.$

Let $\varphi_t:\Sigma\to \Sigma$ be a loop of contactomorphisms that represents the class $\eta,$ and let $h_t:\Sigma\to\R$ be the contact Hamiltonian with respect to the contact form $\alpha$ associated to $\varphi.$ Then, the contact Hamiltonian of $\varphi$ with respect to $\beta$ is equal to $f\cdot h_t:\Sigma\to \R.$ Let $H_t:\hat{W}\to \R$ be a Hamiltonian that is equal to $r\cdot h_t$ on the complement of $\op{int} V\cap\op{int} W$ (this condition makes sure that the time-1 maps of both $H$ and $H\circ \Psi$ are compactly supported in $\op{int} W$). Then, the Hamiltonian $H_t\circ\Psi:W\to\R$ is equal to $\varrho\cdot (f\cdot h_t)$ near the boundary, where $\varrho$ stands for the cylindrical coordinate of the Liouville domain $(W,\mu).$ This implies
\begin{align*}
    &\Theta^\lambda(\eta)= \left[\phi_1^H\right] \in \pi_0\op{Symp}_c (W,d\lambda),\\
    &\Theta^\mu(\eta) = \left[\phi_1^{H\circ\Psi}\right] = \left[\Psi^{-1}\circ\phi_1^H\circ\Psi\right] \in\pi_0\op{Symp}_c (W,d\mu).
\end{align*}
If $\Theta^\lambda(\eta)=0,$ then there exists a symplectic isotopy $\phi_t: (W,d\lambda)\to (W,d\lambda)$ relative to the boundary from the identity to $\phi_1^H.$ Denote by $\phi^\lambda_t: \hat{W}\to\hat{W}$ the flow of the Liouville vector field $X^\lambda$. For $c\in\R^+$ large enough, the symplectomorphism
\[\tilde{\phi}_t:= \left( \phi^\lambda_c \right)^{-1}\circ\phi_t\circ\phi^\lambda_c\]
is compactly supported in $\op{int} V\cap \op{int} W$ for all $t\in [0,1]$. Additionally,
\[s\mapsto \left( \phi^\lambda_{c\cdot s} \right)^{-1}\circ\phi_1\circ\phi^\lambda_{c\cdot s}, \quad s\in[0,1],\]
is a symplectic isotopy, compactly supported in $\op{int} V\cap \op{int} W$, from $\phi_1=\phi_1^H$ to $\tilde{\phi}_1$. Denote by $\overline{\phi}_t$, $t\in[0,1]$, the isotopy that is obtained by concatenating $\tilde{\phi}$ and the inverse of the isotopy above. Then, $\Psi^{-1}\circ\overline{\phi}_t\circ\Psi$ is a symplectic isotopy in $(W,d\mu)$ relative to the boundary from the identity to \mbox{$\Psi^{-1}\circ\phi_1^H\circ\Psi=\phi_1^{H\circ\Psi}.$} Hence, $\Theta^{\mu}(\eta)=0.$ The other direction can be proven similarly.
\end{proof}

Theorem~{\ref{thm:main}} below is a generalisation of Theorem~{\ref{thm:alternatives}} that was stated in the introduction. It uses Proposition~{\ref{prop:floertosingular}} which is stated here and proved in Section~{\ref{sec:morse}}.

\begin{prop}\label{prop:floertosingular}
Let $W$ be a Liouville domain with the boundary $\Sigma:=\partial W,$ and let $h:\Sigma\to\R$ be a contact Hamiltonian such that $0$ is a regular value of $h,$ and such that $h$ has no periodic orbits of period less than or equal to $\varepsilon,$ for some $\varepsilon>0.$ Denote

\begin{align*}
    & \Sigma^-:=\left\{ x\in \Sigma \:|\: h(x)<0  \right\},\\
    & \Sigma^+:=\left\{ x\in \Sigma \:|\: h(x)>0  \right\}.
\end{align*}
Then,
\begin{align*}
    &HF_k(\varepsilon h)\cong H_{k+n}(W, \Sigma^+;\mathbb{Z}_2),\\
    &HF_k(-\varepsilon h)\cong H_{k+n}(W, \Sigma^-;\mathbb{Z}_2).
\end{align*}

\end{prop}

{\begin{theorem}\label{thm:main}
Let $W$ be a Liouville domain such that $c_1(W)=0,$ and let $\varphi_t:\partial W\to \partial W$ be a free contact circle action that is not topologically symmetric with respect to $W.$ Then, $\Theta([\varphi_t])$ is a non-trivial symplectic mapping class in $\pi_0\op{Symp}_c(W).$
\end{theorem}}
\begin{proof}
Denote by $\lambda$ the Liouville form on $W,$ and by $h:\partial W\to \R$ the contact Hamiltonian of $\varphi_t$ with respect to the contact form $\alpha:=\left.\lambda\right|_{\partial W}.$ Lemma~{\ref{lem:strictcontactomorphisms}} on page {\pageref{lem:strictcontactomorphisms}} above implies that there exists a contact form $\alpha'$ on $\partial W$ such that $\varphi_t^\ast\alpha'=\alpha'$ for all $t\in \R$. By this lemma and Lemma~\ref{lem:changealpha}, without loss of generality, we may assume that $\varphi_t$ preserves the contact form $\alpha$ for all $t.$ 

Assume, by contradiction, that $\Theta([\varphi_t])$ is a trivial symplectic mapping class in $\pi_0\op{Symp}_c(W).$ As explained in Section~\ref{sec:les}, $\Theta([\varphi_t])$ is represented by the time-one map, $\phi^H_1,$ of a Hamiltonian $H:\hat{W}\to \R$ that is equal to $r\cdot h$ on the cylindrical end (the Hamiltonian $H$ can be chosen to be autonomous). The inclusion
\[ \op{Symp}_c(W,\lambda)\hookrightarrow \op{Symp}_c(W)\]
is a homotopy equivalence \cite[Lemma~3.3]{uljarevic2017floer}, and $\phi^H_1$ is an exact symplectomorphism that is isotopic to the identity through symplectomorphisms relative to the cylindrical end. Therefore, $\phi^H_1$ is isotopic to the identity through \emph{exact} symplectomorphisms relative to the cylindrical end. Since every isotopy through exact symplectomorphisms is actually a Hamiltonian isotopy, there exists a Hamiltonian $G_t: \hat{W}\to \R$ that is equal to 0 on the cylindrical end, and such that $\phi^G_1=\phi^H_1.$ 

Denote $F:=H\#\overline{G}.$ Let $\varepsilon\in (0,1).$ The Hamiltonian $F$ is equal to $r\cdot h$ on the cylindrical end. This, together with $h$ being a strict contact Hamiltonian, implies
\[\phi^F_t(x,r)= \left(\varphi_t^h(x), r\right),\]
for $r$ big enough. In particular,
\[r\circ \phi_t^F=r,\]
where $r$ is seen as the coordinate function $r: (x,r)\mapsto r$. Since $h$ is autonomous and strict, $h\circ\varphi_t^h=h$ for all $t\in \R$. Hence, the Hamiltonian
\begin{align*}
    \left(\overline{F}\#(\varepsilon\cdot H)\right)_t&=\overline{F}_t + (\varepsilon H_t)\circ \left( \phi^{\overline{F}}_t \right)^{-1}\\
    &= - F_t\circ \phi^F_t + (\varepsilon H_t)\circ\phi^F_t\\
    &= (\varepsilon H_t - F_t)\circ \phi_t^F
\end{align*}
is equal to 
\[(\varepsilon h - h)\cdot r= (\varepsilon -1)\cdot h\cdot r\]
on the cylindrical end.
The naturality isomorphism
\[\mathcal{N}(F): HF(\varepsilon\cdot h)\to HF((\varepsilon-1)\cdot h)\]
is well defined because $F$ generates a loop of Hamiltonian diffeomorphisms. Hence, if $c\in\mathbb{Z}$ denotes the shift in grading,
\[ \dim HF_{k}(\varepsilon\cdot h) =\dim HF_{k+c}((\varepsilon-1)\cdot h), \]
for all $k\in\mathbb{Z}.$ Since the contact Hamiltonian $h$ has no orbits of the period in $(0,\varepsilon)\cup (0, 1-\varepsilon),$ Proposition~\ref{prop:floertosingular} on page~{\pageref{prop:floertosingular}} below implies
\begin{align*}
    & \dim HF_{k}(\varepsilon\cdot h) = \dim H_{k+n}(W,\Sigma^+;\mathbb{Z}_2),\\
    & \dim HF_{k+c}((\varepsilon-1)\cdot h)= \dim H_{k+c+n}(W,\Sigma^-;\mathbb{Z}_2),
\end{align*}
for all $k\in\mathbb{Z}.$ Here,
\begin{align*}
    &\Sigma^+:=\left\{ p\in\Sigma\:|\: h(p)>0 \right\},\\
    &\Sigma^-:=\left\{ p\in\Sigma\:|\: h(p)<0 \right\}.
\end{align*}
Therefore, 
\[\dim H_{k}(W,\Sigma^+;\mathbb{Z}_2)= \dim H_{k+c}(W,\Sigma^-;\mathbb{Z}_2),\]
for all $k\in\mathbb{Z}.$ A generalization of the Lefschetz duality (Theorem 3.43. in \cite{allen2002cambridge}) implies
\[ H_k(W,\Sigma^-;\mathbb{Z}_2)\cong H^{2n-k}(W,\Sigma^+;\mathbb{Z}_2), \]
for all $k\in\mathbb{Z}.$ Consequently,
\begin{align*}
    \dim H_{k}(W,\Sigma^+;\mathbb{Z}_2)&= \dim H_{k+c}(W,\Sigma^-;\mathbb{Z}_2)\\
    &=\dim H^{2n-c-k}(W,\Sigma^+;\mathbb{Z}_2)\\
    &=\dim\op{Hom}\left( H_{2n-c-k}(W,\Sigma^+;\mathbb{Z}_2),\mathbb{Z}_2 \right)\\
    &=\dim H_{2n-c-k}(W,\Sigma^+;\mathbb{Z}_2),
\end{align*}
for all $k\in\mathbb{Z}.$ This contradicts the assumption that $\varphi$ is not topologically symmetric.
\end{proof}

\begin{cor}\label{cor:fundgroup}
Let $W$ be a Liouville domain such that $c_1(W)=0,$ and let $\varphi_t:\partial W\to \partial W$ be a free contact circle action that is not topologically symmetric with respect to $W.$ Then, $\varphi_t$ determines a non-trivial element $[\varphi_t]\in\pi_1\op{Cont}(\partial W).$ In other words, the loop of contactomorphisms $\{\varphi_t\}$ is not contractible.
\end{cor}
\begin{proof}
Since $\Theta$ is a group homomorphism, the triviality of 
\[[ \varphi_t] \in \pi_1 \op{Cont}(\partial W)\]
implies the triviality of $\Theta\left([ \varphi_t]\right) \in \pi_0 \op{Symp}_c(W)$, which contradicts Theorem~\ref{thm:main}.
\end{proof}

\begin{rem}
Theorem~\ref{thm:main} and Corollary~\ref{cor:fundgroup} hold also in the case where $c_1(W)\not=0.$ However, one should then understand the notion of topological symmetry in the following way. Let $N$ be the minimal Chern number of $W,$ and let $\varphi_t:\partial W\to\partial W$ be a free contact circle action with the positive region $\Sigma^+\subset \partial W.$ Denote by $a_j,$ $j\in\mathbb{Z}_{2N}$ the number
\[ a_j:= \sum_{k\equiv j\:(\op{mod} 2N)} \dim H_k(W,\Sigma^+;\mathbb{Z}_2).\]
The contact circle action $\varphi$ is topologically symmetric if there exists $m\in\mathbb{Z}_{2N}$ such that $a_k=a_{m-k},$ for all $k\in\mathbb{Z}_{2N}.$ Proofs of Theorem~{\ref{thm:main}} and Corollary~{\ref{cor:fundgroup}} in the case where $c_1(W)\not=0$ are the same as in the case where $c_1(W)=0$ except at one point, which is discussed next. Let $(H,J)$ be a regular Floer data such that $H:W\to \R$ is a $C^2$ small Morse function. If $c_1(W)\not=0$, the chain complexes $CF_\ast(H,J)$ and $CM_{\ast+n}(H, JX_H)$ are not identical. Namely, the chain complex $HF_\ast(H,J)$ is $\mathbb{Z}_{2N}$-graded whereas $CM_{\ast+n}(H, J X_H)$ is $\mathbb{Z}$-graded. Instead, $HF_\ast(H,J)$ coincides with the $\mathbb{Z}_{2N}$-graded chain complex obtained by ``rolling up'' $CM_{\ast+n}(H, JX_H)$ modulo $2N$. More precisely,
\[CF_j(H,J)=\bigoplus_{k\equiv j\:(\op{mod} 2N)} CM_{k+n}(H, JX_H).\]
Hence, the number $a_j$ above is actually the dimension of the group $HF_{j-n}(\varepsilon\cdot h)$, where $h$ is the contact Hamiltonian of $\varphi_t$ and where $\varepsilon>0$ is a sufficiently small positive number.
\end{rem}

Let us now explain how to derive Theorem~{\ref{thm:alternatives}} from Theorem~{\ref{thm:main}}.

\begin{lem}\label{lem:topsymgen2s}
Let $\varphi_t:\mathbb{S}^{2n-1}\to\mathbb{S}^{2n-1}$ be a free contact circle action on the standard sphere, let $\alpha$ be a contact form on $\mathbb{S}^{2n-1}.$ Denote by $Y$ the vector field of $\varphi_t,$ and by $P$ the set
\[P:=\left\{ p\in\mathbb{S}^{2n-1}\:|\: \alpha(Y(p))>0 \right\}.\]
Then, $\varphi_t$ is topologically symmetric with respect to the standard symplectic ball, $\mathbb{B}^{2n}$, if and only if, there exists $m\in\mathbb{Z}$ such that
\[\dim \tilde{H}_{m-k}(P;\mathbb{Z}_2)=\dim \tilde{H}_{k}(P;\mathbb{Z}_2)\]
for all $k\in\mathbb{Z}.$ Here, $\tilde{H}$ stands for the reduced singular homology.
\end{lem}
\begin{proof}
Since $\mathbb{B}^{2n}$ is contractible, from the long exact sequence for the reduced singular homology of the pair $(\mathbb{B}^{2n}, P)$ we deduce 
\[ \tilde{H}_{k}(P;\mathbb{Z}_2) \cong H_{k+1}(\mathbb{B}^{2n}, P;\mathbb{Z}_2), \] 
for all $k\in\mathbb{Z}$. The contact Hamiltonian that generates $\varphi_t$ is given by $h=-\alpha(Y)$, and therefore, the set $P$ is the negative region of the contact circle action $\varphi$ in the sense of Definition~\ref{def:topsym}. The topological symmetry of $\varphi_t$ with respect to $\mathbb{B}^{2n}$ is equivalent to the existence of $m\in\mathbb{Z}$ such that 
\[ \dim H_{m-k}(\mathbb{B}^{2n}, P; \mathbb{Z}_2) = \dim H_k(\mathbb{B}^{2n}, P; \mathbb{Z}_2),  \] 
for all $k\in\mathbb{Z}$. Using the above mentioned isomorphism, this is further equivalent to 
\[ \dim \tilde{H}_{m-2-(k-1)}(P;\mathbb{Z}_2) = \dim \tilde{H}_{k-1}(P;\mathbb{Z}_2), \] 
for all $k\in\mathbb{Z}$, which finishes the proof.
\end{proof}

\settheoremtag{\ref{thm:alternatives}}
\begin{theorem}
For all $n\in\mathbb{N},$ at least one of the following statements is true.
\begin{enumerate}[A]
    \item There exists an exotic symplectomorphism of $\mathbb{B}^{2n}$ (i.e. there exists a non-trivial element of $\pi_0\op{Symp}_c(\mathbb{B}^{2n})$).
    \item Every free contact circle action on $\mathbb{S}^{2n-1}$ is topologically symmetric.
\end{enumerate}
\end{theorem}
\begin{proof}
This is a direct consequence of Theorem~\ref{thm:main} and Lemma~{\ref{lem:topsymgen2s}}.
\end{proof}

\section{Morse theory}\label{sec:morse}
In this section, we prove the main technical result that computes the Floer homology for a contact Hamiltonian with sufficiently small absolute value. Using the standard argument, one can reduce the Floer homology to the Morse homology of a function on a manifold with non-empty boundary. The Morse theory for manifolds with boundary has been intensively studied \mbox{\cite{morse1934critical,jankowski1972functions,braess1974morse,kronheimer2007monopoles,laudenbach2011morse}}\nocite{hajduk1981minimal,barannikov1994framed,cornea2003rigidity,hajduk1981minimal,barannikov1994framed,cornea2003rigidity,kronheimer2007monopoles,bloom2012combinatorics,borodzik2016morse,pushkar2019morse}. A Morse function whose restriction to the boundary is also Morse is called an m-function. Given an m-function $f:W\to\R,$ it is known that the critical points of $f$ together with some of the critical points of $\left. f\right|_{\partial W}$ recover the singular homology of $W$ \cite{kronheimer2007monopoles, laudenbach2011morse}. Whether a critical point of $\left. f\right|_{\partial W}$ will be taken into account or not is determined by the direction in which the gradient $\nabla f$ points at that point. Namely, a critical point $p$ of $\left. f\right|_{\partial W}$ will be ignored if, and only if, the gradient $\nabla f(p)$ points outwards.

The results in the literature do not cover entirely the Morse theory required in the proof of Proposition~\ref{prop:floertosingular}. For instance, the proof will deal with Morse functions whose restrictions to the boundary are not necessarily Morse. For the convenience of the reader, we recall the statement of Proposition~{\ref{prop:floertosingular}}.

\begin{taggedtheorem}{\ref{prop:floertosingular}}
Let $W$ be a Liouville domain with the boundary $\Sigma:=\partial W,$ and let $h:\Sigma\to\R$ be a contact Hamiltonian such that $0$ is a regular value of $h,$ and such that $h$ has no periodic orbits of period less than or equal to $\varepsilon,$ for some $\varepsilon>0.$ Denote
\begin{align*}
    & \Sigma^-:=\left\{ x\in \Sigma \:|\: h(x)<0  \right\},\\
    & \Sigma^+:=\left\{ x\in \Sigma \:|\: h(x)>0  \right\}.
\end{align*}
Then,
\begin{align*}
    &HF_k(\varepsilon h)\cong H_{k+n}(W, \Sigma^+;\mathbb{Z}_2),\\
    &HF_k(-\varepsilon h)\cong H_{k+n}(W, \Sigma^-;\mathbb{Z}_2).
\end{align*}
\end{taggedtheorem}
\begin{proof} The proof is divided into several steps. In the first step, we pass from Floer to Morse homology. Steps 2-7 reduce the proof (by pasting and cutting) to the known case where the Morse function is constant on the boundary components. Figure~\ref{fig:bigfig} on page \pageref{fig:bigfig} illustrates the proof. 

Lemma~{\ref{lem:modify}} below states that the groups $HF_\ast(a\cdot h)$ and $HF_\ast(b\cdot h)$ are isomorphic provided that the contact Hamiltonian $h$ has no closed orbits of period in $[a,b]$. Hence, we may assume, without loss of generality, that $\varepsilon>0$ is arbitrary small.

\textbf{Step~1} (Passing to the Morse homology). Let $H:\hat{W}\to\R$ be a Morse function such that $H(x,r)=h(x)\cdot r$ for $(x,r)\in\Sigma\times[1,\infty).$ For $\varepsilon>0$ small enough, there exists an almost complex structure $J$ on $\hat{W}$ such that $(\varepsilon\cdot H, J)$ is Floer data for $\varepsilon\cdot h$ and such that the pair $(\varepsilon\cdot H, \varepsilon\cdot JX_H)$ is Morse-Smale \cite[Chapter~10]{audin2014morse} (note that we are using different conventions from \cite{audin2014morse}, namely we define $X_H$ as the vector field that satisfies $dH=d\lambda(X_H,\cdot)$). Moreover,
\[CF_\ast(\varepsilon\cdot H,J)= CM_{\ast+n}(\varepsilon\cdot H, \varepsilon\cdot JX_H),\]
where $CM_\ast$ stands for the Morse complex. Consequently,
\[\dim HF_k(\varepsilon h)=\dim HM_{k+n}(\varepsilon\cdot H, \varepsilon\cdot JX_H),\]
where $HM_\ast$ is the Morse homology.

\textbf{Step~2} (Extending the Liouville domain). Let 
\begin{align*}
    &L:=\max\{\varepsilon\cdot H(p)\:|\: p\in\op{Crit} H\},\\
    &\ell:=\min\{\varepsilon\cdot H(p)\:|\: p\in\op{Crit} H\}.
\end{align*}
In this step, we prove that there exists $R>0$ such that the extension
\[W(R):=\hat{W}\setminus (\Sigma\times(R,\infty))\]
of the Liouville domain $W$ has the following property: the critical values of $\left.\varepsilon\cdot H\right|_{\partial W(R)}$ do not lie in the interval $[\ell,L].$ Since 0 is a regular value of $\varepsilon\cdot h$ and since the domain of $\varepsilon\cdot h,$ $\Sigma,$ is compact, there exists $\delta>0$ such that $\varepsilon\cdot h$ has no critical values in $(-\delta,\delta).$ A point $p\in\Sigma$ is a critical point of $\varepsilon\cdot h$ if, and only if, $(p,R)\in\partial W(R)=\Sigma\times\{R\}$ is a critical point of $\left.\varepsilon\cdot H\right|_{\partial W(R)}.$ Hence $\left.\varepsilon\cdot H\right|_{\partial W(R)}$ has no critical values in $(-\delta R, \delta R).$ If $R$ is big enough, then $[\ell,L]$$\subset (-\delta R, \delta R).$ See Figure~{\ref{fig:bigfig}}, Step~2 on page~{\pageref{fig:bigfig}}.

\textbf{Step~3} (Constructing the double manifold). In this step, we construct a closed manifold (not necessarily symplectic), $M$, by gluing two copies of $W(R)$ along the boundary. See Figure~{\ref{fig:bigfig}}, Step~3 on page~{\pageref{fig:bigfig}}. Denote by $W_A$ and $W_B$ those two copies. Explicitly,
\[ M:= (W_A\sqcup (\Sigma\times\R)\sqcup W_B)/\sim, \]
where $\sim$ stands for the following identifications
\begin{align*}
    & W_A\quad\supset\quad \Sigma\times(0, R] \to \Sigma\times\R\quad:\quad (x,r)\mapsto \left(x, \log\left( \frac{r}{R} \right) \right),\\
    & W_B\quad\supset\quad \Sigma\times(0, R] \to \Sigma\times\R\quad:\quad (x,r)\mapsto \left(x, -\log\left( \frac{r}{R} \right) \right).
\end{align*}
Let $F:M\to \R$ be a function that is obtained by smoothing out the function on $M$ equal to $\varepsilon\cdot H$ on both $W_A$ and $W_B.$ A formal definition of $F$ follows.
Let $\chi:\R\to$ $\R_{<0}$ be a smooth concave function such that $\chi(s)=s$ for $s\leqslant \log\left( \frac{R-1}{R} \right),$ $\chi(s)=-s$ for $s\geqslant -\log\left( \frac{R-1}{R} \right),$ and such that $\chi$ has a unique maximum at $s=0.$ The function $\chi$ can be seen as a smoothening of the function $s\mapsto -\abs{s}$ by a compactly supported perturbation. Denote by $F:M\to\R$ the function defined by
\[F(p):=\left\{ \begin{tabular}{ll}
     $\varepsilon \cdot H(p)$ & \text{for } $p\in W(R-1)\subset W_A,$  \\
     $R\cdot e^{\chi(s)}\cdot \varepsilon \cdot h(y)$&\text{for } $p=(y,s)\in \Sigma\times\R,$\\
     $\varepsilon \cdot H(p)$ & \text{for } $p\in W(R-1)\subset W_B.$
\end{tabular} \right.\]

\begin{figure}[ht]
    \centering
    \def\svgwidth{\columnwidth}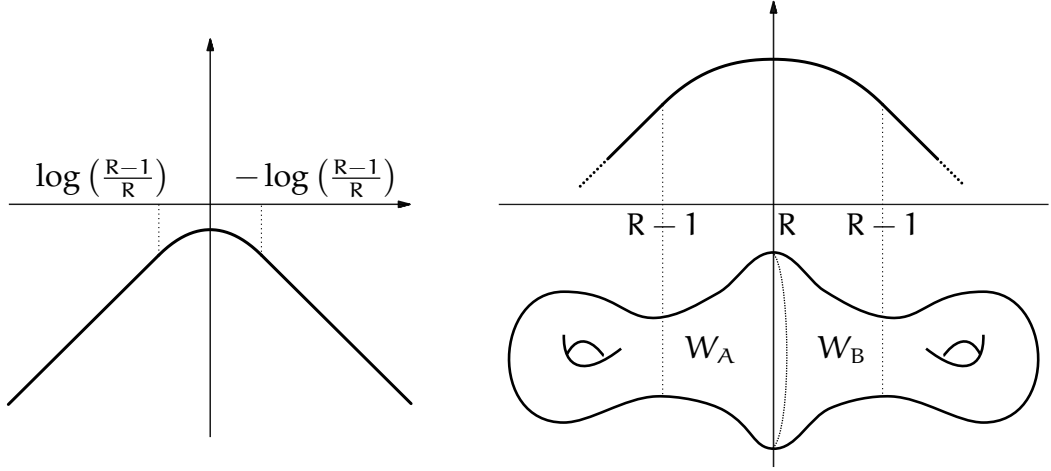
    \caption{An illustration of the functions $\chi$ (left) and $F$ (right). The function $F$ is graphed along $M$ in the case of a constant contact Hamiltonian $h$.} 
\end{figure}

\textbf{Step~4} (Truncated double manifold). There are three types of critical points of the function $F : $ the critical points of $H$ in $W_A,$ the critical points of $H$ in $W_B,$ and the critical points on $\Sigma\times\{0\}\subset \Sigma\times\R$ that correspond to the critical points of $h:\Sigma\to\R.$ As explained in Step~2, for $R>0$ big enough, the critical points of the third type have values outside the interval $[\ell, L].$ In fact, one can choose $R$ big enough so that the critical points of the third type have values outside an interval $[-K,K]\supset [\ell,L]$ where $K$ satisfies
\[ F^{-1}(-K), F^{-1}(K)\subset \Sigma\times\R\subset M. \]
Then, $-K$ and $K$ are regular values of $F$ and the only critical points of $F$ in $F^{-1}([-K, K])$ are the ones of type 1 and type 2. Note that these critical points are nondegenerate. Hence $F^{-1}([-K, K])$ is a manifold with boundary whose boundary components are regular level sets of $F,$ and $F$ is a Morse function on $F^{-1}([-K, K]).$ Denote $M_T:=F^{-1}([-K, K])$ (see Figure~{\ref{fig:bigfig}}, Step~4 on page~{\pageref{fig:bigfig}}).

\textbf{Step~5} (No crossing). Let $g$ be a Riemannian metric on $M$ such that $g=d\lambda(\cdot, J\cdot)$ on both $W(R/2)\subset W_A$ and $W(R/2)\subset W_B.$ Let $X=\nabla F,$ where $\nabla$ is the gradient with respect to $g.$ In particular, if $R>2,$ $X$ is equal to $\varepsilon \cdot J X_H$ on both $W(R/2)\subset W_A$ and $W(R/2)\subset W_B.$ In this step, we prove that for $R$ large enough, there are no integral curves of $X$ that connect two critical points of $F$ in $M_T$ and intersect the submanifold $\Sigma\times\{0\}\subset \Sigma\times\R\subset M.$ This eliminates the integral curves of $X$ that connect a critical point of $F$ in $W_A$ with a critical point of $F$ in $W_B,$ and also the integral curves of $X$ that connect two critical points of $F$ in $W_A$ (or $W_B$) but at some point leave $W_A$ ($W_B,$ respectively). 

Assume there exists an integral curve $\gamma:\R\to M$ of $X$ such that $\gamma$ connects two critical points of $F$ in $M_T$ and such that it is not contained in one of the regions $W_A$ and $W_B.$ Without loss of generality assume $\displaystyle \lim_{t\to-\infty}\gamma(t)\in W_A.$ Then, for $R>2,$ there exists an interval $[a,b]\subset\R$ such that $ \gamma(a)\in\Sigma\times\{1\}\subset W_A$ and $\gamma(b)\in\Sigma\times \{R/2\}\subset W_A.$ Since
\begin{align*}
    L-\ell &\geqslant F(\lim_{t\to+\infty}\gamma(t))- F(\lim_{t\to-\infty}\gamma(t))\\
    &= \int_{-\infty}^{+\infty} dF(\dot{\gamma}(t))dt\\
    &= \int_{-\infty}^{+\infty} g\left(\nabla F(\gamma(t)), \dot{\gamma}(t)\right)dt\\
    &= \int_{-\infty}^{+\infty} g\left(X(\gamma(t)), X(\gamma(t))\right)dt\\
    &\geqslant \int_a^b g(\varepsilon \cdot JX_H(\gamma(t)), \varepsilon \cdot JX_H(\gamma(t)))dt\\
    &= \varepsilon^2\cdot \int_a^b g(X_H(\gamma(t)), X_H(\gamma(t)))dt\\
    &\geqslant \varepsilon^2\cdot(b-a)\cdot \min\left\{ \norm{X_H(p)}^2_g\:|\: p\in \Sigma\times\left[1,\frac{R}{2}\right]\right\},
\end{align*}
we have
\[b-a\leqslant \frac{L-\ell}{\varepsilon^2\cdot\min\left\{ \norm{X_H(p)}^2_g\:|\: p\in \Sigma\times\left[1,\frac{R}{2}\right]\right\}}.\]
Since $h$ is a strict contact Hamiltonian (i.e. it generates an isotopy that preserves the contact form), the vector $X_H(x,r)$ is independent of $r.$ Hence
\[ \min\left\{ \norm{X_H(p)}^2_g\:|\: p\in \Sigma\times\left[1,\frac{R}{2}\right]\right\}= \min_{p\in\Sigma\times\{1\}}\norm{X_H(p)}_g^2, \]
and, consequently, $b-a<C,$ where $C\in\R^+$ is a constant that does not depend on $R.$ For $r\in\left[1,\frac{R}{2}\right]$, the vector field $\partial_r$ is orthogonal (with respect to $g$) to $\Sigma\times\{ r \}$ and $\norm{\partial_r}_g=\frac{1}{\sqrt{r}}.$ The latter follows because
\begin{align*}
    \norm{\partial_r}_g^2&= d\lambda (\partial_r, J\partial_r)\\
    &= dr\wedge\alpha(\partial_r, J\partial_r) + r\cdot d\alpha (\partial_r, J\partial_r)\\
    &=dr(\partial_r) \cdot \alpha(J\partial_r) - \alpha(\partial_r)\cdot dr(J\partial_r)\\
    &=\alpha(J\partial_r)\\
    &=\frac{1}{r}\cdot (\lambda\circ J)(\partial_r)\\
    &= \frac{1}{r}\cdot (-dr\circ J\circ J)(\partial_r)\\
    &= \frac{1}{r}.
\end{align*}

The orthogonal projection (with respect to $g$) of the vector $\dot{\gamma}(t)$ to the 1-dimensional vector space
\[\{s\cdot \partial_r\:|\: s\in\R\}\]
is equal to $\frac{d}{dt}\Big(r(\gamma(t))\Big)\cdot \partial_r$. Hence, the Pythagorean theorem implies
\begin{align*}
    \norm{\dot{\gamma}(t)}^2_g &\geqslant \abs{\frac{d}{dt}\bigg(r(\gamma(t))\bigg)}^2\cdot \norm{\partial_r}_g^2\\
    &= \abs{\frac{d}{dt}\bigg(r(\gamma(t))\bigg)}^2 \cdot \frac{1}{r(\gamma(t))}\\
    &= \abs{\frac{d}{dt}\left( 2\sqrt{ r(\gamma(t))} \right)}^2.
\end{align*}

Therefore,
\begin{align*}
    L-\ell &\geqslant \int_a^b \norm{\dot{\gamma}(t)}^2_g dt\\
    &\geqslant\int_a^b \abs{\frac{d}{dt}\bigg(2\sqrt{r(\gamma(t))}\bigg)}^2 dt\\
    &\geqslant \frac{1}{b-a}\cdot \left(\int_a^b \frac{d}{dt}\bigg(2\sqrt{r(\gamma(t))}\bigg)dt\right)^2\\
    &\geqslant \frac{1}{b-a}\cdot\left( 2\sqrt{r(\gamma(b))}-2\sqrt{r(\gamma(a))}\right)^2\\
    &\geqslant \frac{1}{C}\cdot \left(\sqrt{2R}-2\right)^2.
\end{align*}

This, however, cannot be possible for $R$ big enough. In other words, for $R$ big enough, there are no integral curves of $X$ that cross $\Sigma\times\{0\}\subset \Sigma\times\R\subset M$ and connect critical points of $F$ in $M_T.$ In fact, we proved that such integral curves are all contained in either $W(R/2)\subset W_A$ or $W(R/2)\subset W_B.$ As a consequence, $F$ and $X$ satisfy the Smale condition in $M_T,$ and
\[ CM_\ast (F, X) = CM_\ast(\varepsilon\cdot H, \varepsilon\cdot JX_H)\oplus CM_\ast(\varepsilon\cdot H, \varepsilon\cdot JX_H). \]
Hence
\[ 2\dim HF_k(\varepsilon\cdot h)=\dim HM_{k+n}(F,X). \]
Denote by $\partial^+M_T$ the part of the boundary $\partial M_T$ on which the vector field $X$ points inwards. Similarly, denote $\partial^-M_T:=\partial M_T\setminus \partial^+ M_T,$ i.e. $\partial^-M_T$ is the part of the boundary $\partial M_T$ on which the vector field $X$ points outwards. Alternatively, $\partial^+M_T$ and $\partial^-M_T$ can be defined by
\[\partial^+M_T:= F^{-1}(-K),\quad \partial^-M_T:=F^{-1}(K).\]
Since the boundary components of the manifold $M_T$ are regular level sets of $F,$ the Morse homology of $F:M_T\to \R$ can be expressed as the singular homology of a pair (see, for instance, \cite[Theorem~3.9]{hutchings2002lecture})
\[HM_\ast(F, X)\cong H_\ast(M_T, \partial^- M_T).\]

\textbf{Step~6} (Mayer-Vietoris long exact sequence). Denote $V_A:= W_A\cap M_T,$ $V_B:=W_B\cap M_T,$ $U_A:= \partial^- M_T\cap W_A,$ and $U_B:=\partial^-M_T\cap W_B.$ The relative form of the Mayer-Vietoris long exact sequence implies that there exists a long exact sequence
\begin{equation*}
\begin{tikzcd}
\vdots \arrow{d} &\arrow[draw=none]{d} \\
H_k(V_A\cap V_B, U_A\cap U_B) \arrow{d} & H_{k-1}(V_A\cap V_B, U_A\cap U_B) \arrow{d}\\
 H_k(V_A, U_A)\oplus H_k(V_B, U_B) \arrow{d}\arrow[draw=none]{r}[name=Z, shape=coordinate]{} & \vdots\\
H_k(M_T, \partial^- M_T)\arrow[rounded corners, to path={--([yshift=-2ex]\tikztostart.south) -| ([xshift=-4ex] Z) |-  ([yshift=2ex]\tikztotarget.north) [near start]\tikztonodes -- (\tikztotarget)}]{ruu}
\end{tikzcd}
\end{equation*}

The function $\left.F\right|_{V_A\cap V_B}$ is a Morse function that has no critical points. Hence $V_A\cap V_B$ is a trivial cobordism, i.e. $V_A\cap V_B\approx (U_A\cap U_B)\times[0,1].$ In particular,
\[ H_k(V_A\cap V_B, U_A\cap U_B)=0 \]
for all $k\in\mathbb{Z}.$ Consequently,
\[H_k(M_T, \partial^- M_T)\cong H_k(V_A, U_A)\oplus H_k(V_B, U_B) \]
for all $k\in\mathbb{Z}.$ Since the pairs $(V_A, U_A)$ and $(V_B, U_B)$ are homeomorphic, the following holds
\[(\forall k\in\mathbb{Z})\quad \dim H_k(M_T, \partial^- M_T)= 2\dim H_k(V_A, U_A). \]
Therefore,
\[(\forall k\in\mathbb{Z})\quad \dim HF_k(\varepsilon h)= \dim H_{k+n}(V_A, U_A).\]
See Figure~{\ref{fig:bigfig}}, Step~6 on page~{\pageref{fig:bigfig}}.

\textbf{Step~7.} In this step, we prove that the pair $(V_A, U_A)$ is homeomorphic to the pair $(W,\Sigma^+).$ The boundary of $V_A$ consists of parts of the submanifolds $\Sigma\times\{0\}\subset \Sigma\times \R\subset M,$ $F^{-1}(-K),$ and $F^{-1}(K).$ These parts belong to $\Sigma\times\R\subset M.$ Recall that, on $\Sigma\times\R\subset M$, the function $F$ is given by
\[(y, s)\mapsto R\cdot e^{\chi(s)}\cdot \varepsilon\cdot h(y).\] The function $F$ does not change the sign along the curve $\{x\}\times\R$, for $x\in\Sigma.$ Therefore, there does not exist $x\in\Sigma$ such that $\{x\}\times\R$ intersects both of the sets
\[(\partial V_A)\cap F^{-1}(-K), \quad (\partial V_A)\cap F^{-1}(K). \] 
 The derivative 
\[\frac{d}{ds} \left( Re^{\chi(s)}\varepsilon h(y) \right)= \varepsilon Re^{\chi(s)} \chi'(s)h(y)\]
is positive for positive $h(y)$ and negative for negative $h(y)$ on $\Sigma\times(-\infty, 0]$. Therefore, the value of $F$ on the curve $\{x\}\times(-\infty,0]$ belongs to $(-K,K)$ if
\[(x,0)\in \op{int}\bigg( V_A \cap (\Sigma\times\{0\}) \bigg) = (\Sigma\times\{0\})\cap F^{-1}(-K,K).\]
Consequently, the curve $\{x\}\times\R$, $x\in \Sigma$, cannot intersect two of the sets
\[(\partial V_A)\cap F^{-1}(-K), \quad (\partial V_A)\cap (\Sigma\times\{0\}),\quad (\partial V_A)\cap F^{-1}(K)\]
except at the points where those two sets intersect themselves.

Since
\[ \frac{d}{ds} \left( Re^{\chi(s)}\varepsilon h(y) \right)= \varepsilon Re^{\chi(s)} \chi'(s)h(y)\not=0  \]
for $s\not=0,$ and since the curve $\{x\}\times\R$ cannot intersect different parts of $\partial V_A$ (except at the points where they intersect) for all $x\in\Sigma,$ the boundary of $V_A$ can be seen as the graph in $\Sigma\times(-\infty,0]$ of a continuous (in fact, a piecewise-smooth) function. Lemma~{\ref{lem:abovethegraph}} on page~{\pageref{lem:abovethegraph}} below states that in this case $V_A= W_A\cap F^{-1}([-K,K])$ is homeomorphic to $W_A$ via a homeomorphism that sends $U_A= F^{-1}(K)\cap W_A$ to $F^{-1}([K,+\infty))\cap \partial W_A.$ The continuous function $f$ to which Lemma~{\ref{lem:abovethegraph}} is applied is given by 
\[f(y):=\min\left\{\left(F^{-1}_y\right)(K), 0\right\},\]
where $F_y:\R\to\R: s\mapsto F(y,s).$
The homeomorphism furnished by Lemma~{\ref{lem:abovethegraph}} maps indeed $U_A$ to $F^{-1}([L, +\infty))\cap \partial W_A$ for the following reason. For all $(y, f(y))\in U_A$, either $f(y)=0$ or $0>f(y)= F^{-1}_y(K)$ (which implies $F(y,0)>K$).

Since there are no critical points of $\left.F\right|_{\partial W_A}$ in $F^{-1}([-K, K]),$ the pair
\[\left(W_A, \partial W_A \cap  F^{-1}([K, +\infty)) \right)\]
is homeomorphic (in fact, diffeomorphic) to the pair $(W_A, \partial W_A\cap \{F\geqslant 0\}).$ Consequently, the pair $(V_A, U_A)$ is homeomorphic to the pair $(W, \{h\geqslant 0\}).$ Hence
\begin{align*}
    H_\ast (V_A, U_A)&\cong H_\ast (W, \{h\geqslant 0\})\\
    &\cong H_\ast (W, \{h> 0\})\\
    &= H_\ast (W, \Sigma^+).
\end{align*}
See Figure~{\ref{fig:bigfig}}, Step~7 on page~{\pageref{fig:bigfig}}.
\end{proof}

\begin{figure}
\caption{An illustration of the proof of Proposition~\ref{prop:floertosingular}}
\label{fig:bigfig}
\vspace{1cm}
\centering
\includegraphics{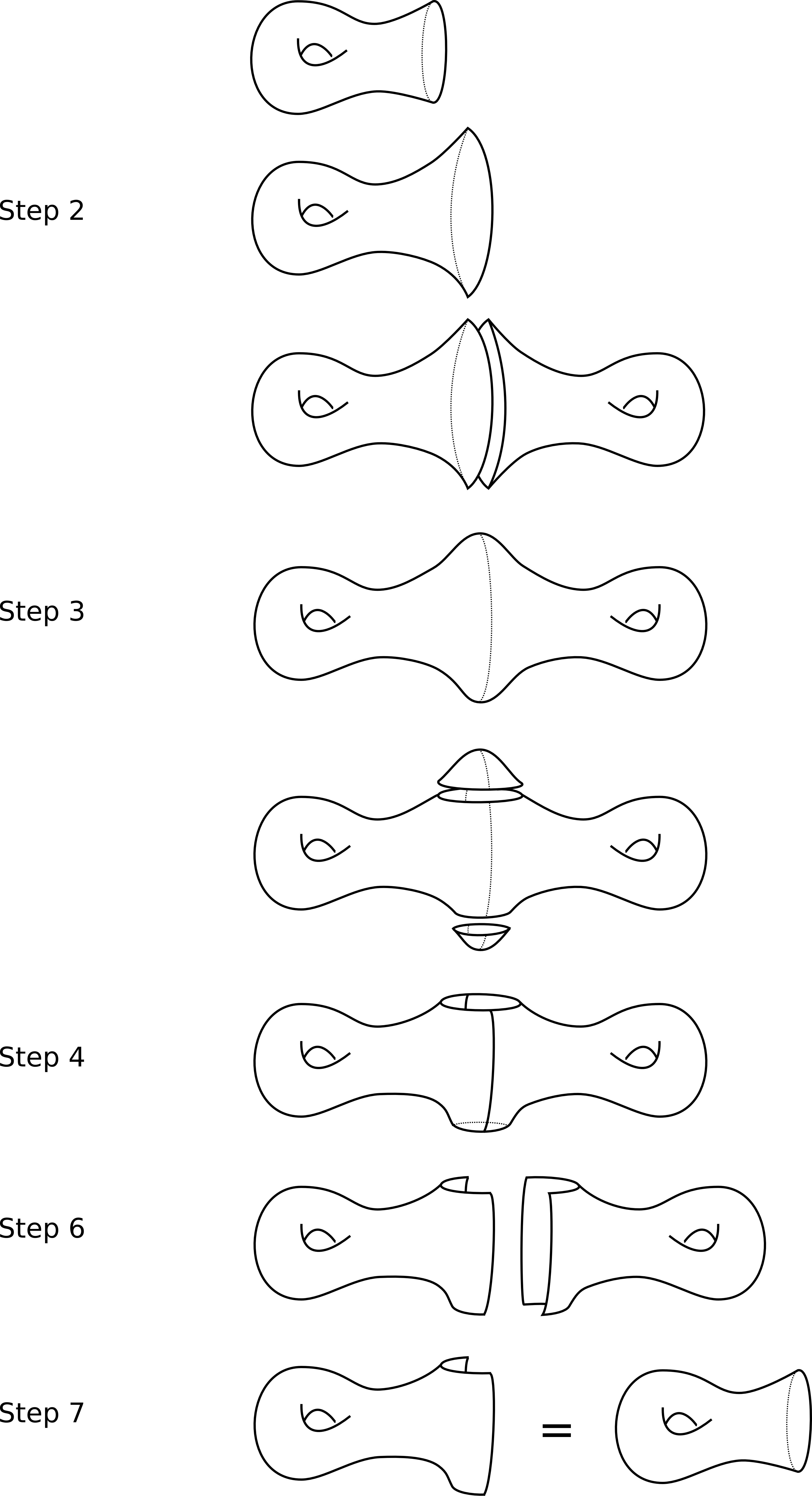}
\end{figure}

The following lemma was used at the beginning of the proof of Proposition~{\ref{prop:floertosingular}}. It allowed passing from Floer to Morse homology.
\begin{lem}\label{lem:modify}
Let $W$ be a Liouville domain with the boundary $\Sigma:=\partial W,$ and let $h:\Sigma\to \R$ be a strict contact Hamiltonian (i.e.  $dh(R^\alpha)=0$ where $R^\alpha$ is the Reeb vector field). Assume that the contact Hamiltonian $c\cdot h$ has no 1-periodic orbits for $c\in[a,b].$ Then, the groups $HF_\ast(a\cdot h)$ and $HF_\ast(b\cdot h)$ are isomorphic.
\end{lem}
\begin{proof}
We prove here that there exists a smooth family $H^s:\Sigma\times [1,+\infty)\to\R,\:s\in[0,1]$ of Hamiltonians without 1-periodic orbits such that
\begin{itemize}
    \item $H^0(x,r)= a\cdot h(x)\cdot r,$ for $(x,r)\in\Sigma\times [1,+\infty),$
    \item $H^s(x,r)= a\cdot h(x)\cdot r,$ for $s\in[0,1]$ and for $r$ in a neighbourhood of 1,
    \item $H^1(x,r)= b\cdot h(x)\cdot r,$ for $r$ big enough.
\end{itemize}
Informally, this means that one can modify the slope smoothly on the cylindrical end from $a\cdot h$ to $b\cdot h$ without creating any additional 1-periodic orbits. Hence, it is possible to find Floer data $(H^a, J^a)$ and $(H^b, J^b)$ that compute $HF_\ast(a\cdot h)$ and $HF_\ast(b\cdot h)$, respectively, and such that the chain complexes $CF_\ast(H^a, J^a)$ and $CF_\ast(H^b, J^b)$ are identical. As opposed to the case where $h$ is constant, the proof is not direct.

Denote by $\mathfrak{X}(\Sigma)$ the space of smooth vector fields on $\Sigma$ endowed with the $C^1$ topology, and denote by $\phi^X_t:\Sigma\to\Sigma$ the flow of a vector field $X\in \mathfrak{X}(\Sigma).$ Denote by $Y^h$ the vector field of the contact isotopy furnished by the contact Hamiltonian $h$. The map
\[\mathfrak{X}(\Sigma)\times\Sigma\to\Sigma\quad:\quad (X,p)\mapsto \phi^X_t(p)\]
is continuous for all $t\in\R.$ This (together with the time-1 map of the flow of $c\cdot Y^h$ not having any fixed points for $c\in[a,b]$) implies that there exists an open neighbourhood $U\subset \mathfrak{X}(\Sigma)$ of 
\[\{c\cdot Y^h\:|\:c\in[a,b]\}\]
such that $\phi^Z_1$ has no fixed points for all $Z\in U.$

Let $\delta>0,$ and let $\mu:\R^+\to\R$ be a smooth function such that
\begin{itemize}
    \item $\mu(r)= a\cdot r,$ for $r<2,$
    \item $\mu(r)= b\cdot r,$ for $r$ sufficiently large,
    \item $\frac{\mu(r)}{r}\in[a,b],$ for all $r\in\R^+,$
    \item $\abs{\mu'(r)-\frac{\mu(r)}{r}}<\delta,$ for all $r\in\R^+.$
\end{itemize}
(Such a function $\mu$ can be constructed in the following way. Let $\kappa:\R^+\to\R$ be a compactly supported smooth function such that $\op{supp} \kappa\subset (2,+\infty),$ such that $\kappa(r)\leqslant\frac{\delta}{r}$ for all $r\in\R^+,$ and such that $\int_{\R} \kappa(r)dr=b-a.$ The conditions on $\kappa$ are not contradicting each other because $\int_2^{+\infty}\frac{\delta}{r}dr=+\infty,$ and therefore, $\int_{\R}\kappa(r)dr$ can be chosen arbitrary large without violating the condition $\kappa(r)\leqslant\frac{\delta}{r}.$ The function $\mu(r):=a\cdot r + r\cdot \int_0^{r}\kappa(s)ds$ satisfies the conditions above. See Figure~{\ref{fig:kappaMuFunctions}} for graphs of $\kappa$ and $\mu$.)

\begin{figure}[ht]
    \centering
    \includegraphics[scale=0.2]{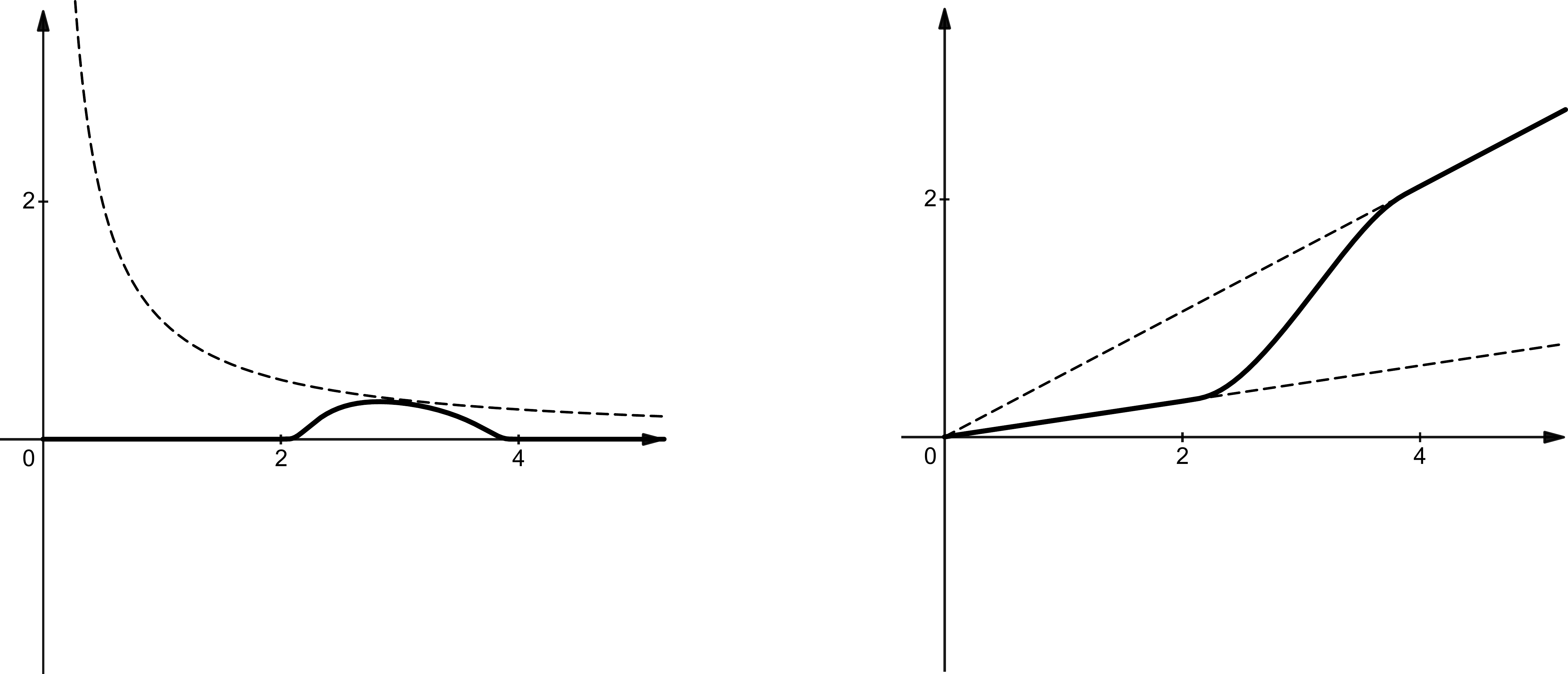}
    \caption{An illustration of the functions $\kappa$ (left) and $\mu$ (right).} \label{fig:kappaMuFunctions}
\end{figure}

We will show that when $\delta>0$ is sufficiently small, the homotopy
\[ H^s: \Sigma\times[1,+\infty)\to \R\quad:\quad (x,r)\mapsto \left((1-s)\cdot a\cdot r + s\cdot \mu(r)\right)\cdot h(x)\]
satisfies the conditions from the beginning of this proof. The only non-trivial condition to check is that $H^s$ has no 1-periodic orbits for all $s\in[0,1].$ The vector field of the Hamiltonian $H^s$ is equal to
\[ X_{H^s}(x,r)= \left( (1-s)\cdot a +s\cdot \frac{\mu(r)}{r} \right)\cdot Y^h(x) + s\cdot\left( \frac{\mu(r)}{r} - \mu'(r) \right)\cdot h(x)\cdot R(x), \]
where $R=R^\alpha$ is the Reeb vector field on $\Sigma$ (the computation used $dh(R)=0$). In particular, the flow of $H^s$ preserves the submanifolds $\Sigma\times \{r\}, r\in[1,+\infty).$ Therefore, it is enough to prove that the restriction of the flow of $H^s$ to $\Sigma\times\{r\}$ has no 1-periodic orbits for all $r\in[1,+\infty)$ and all $s\in[0,1].$ By the assumptions, $(1-s)\cdot a +s\cdot \frac{\mu(r)}{r}\in[a,b]$ for all $s\in[0,1]$ and $r\in[1, +\infty).$ Hence, for $\delta>0$ sufficiently small, $X_{H^s}(\cdot, r)\in U$ for $s\in[0,1]$ and $r\in[1,+\infty).$ Consequently, for $\delta>0$ sufficiently small, the flow of $X_{H^s}$ has no 1-periodic orbits.
\end{proof}

The next lemma is a topological fact that was used in the final step in the proof of Proposition~{\ref{prop:floertosingular}} above.

\begin{lem}\label{lem:abovethegraph}
Let $X$ be a compact topological space, let $\varepsilon\in\R^+$ be a positive real number, and let $f:X\to[0,+\infty)$ be a continuous function. Then, there exists a homeomorphism
\[\psi:X\times [0,+\infty)\to \left\{ (x,r)\in X\times [0,+\infty)\:|\: r\geqslant f(x) \right\}\]
such that $\psi(x,0)=(x, f(x))$ for all $x\in X,$ and such that $\psi(x,r)=(x,r)$ if $r\geqslant f(x)+\varepsilon.$
\end{lem}
\begin{proof}
Denote 
\[Y:=\left\{ (x,r)\in X\times [0,+\infty)\:|\: r\geqslant f(x) \right\}.\]
A homeomorphism satisfying the conditions of the lemma can be constructed explicitly as follows. Let
\[\psi(x,r):=\left\{ \begin{tabular}{c l}
     $\left(x, f(x)+ r\cdot \frac{\varepsilon}{f(x)+\varepsilon} \right)$& \text{for }$r\in[0, f(x)+\varepsilon]$  \\
     $(x,r)$& for $r\in[f(x)+\varepsilon,+\infty).$ 
\end{tabular} \right.\]
The function $\psi:X\times[0,+ \infty)\to Y$ is well defined and continuous. The function
\[(y,s)\mapsto \left\{ \begin{tabular}{c l}
     $\left(y, (s-f(y))\cdot \frac{f(y)+\varepsilon}{\varepsilon}\right)$& \text{for }$s\in[f(y), f(y)+\varepsilon]$  \\
     $(y,s)$& for $s\in[f(y)+\varepsilon,+\infty).$ 
\end{tabular} \right. \]
is a well-defined continuous function $Y\to X\times [0,+\infty)$, and it is inverse to $\psi.$ Hence $\psi$ is a homeomorphism. 
\end{proof}

\printbibliography
\end{document}